\documentclass[12pt]{article}
\usepackage[T2A]{fontenc}
\usepackage[english]{babel}
\usepackage[tbtags]{amsmath}
\usepackage{amsfonts,amssymb}
\usepackage{amssymb}
\usepackage{graphicx}
\usepackage{amscd}

\usepackage{fullpage}
\usepackage{float}

\usepackage{graphics,amsmath,amssymb}
\usepackage{amsthm}
\usepackage{amsfonts}
\usepackage{mathrsfs}

\title{The method for obtaining expressions for coefficients of reverse generating functions}

\author{V.~V.~Kruchinin\\
\small Tomsk State University of Control Systems and Radioelectronics, Russian Federation\\
\small \texttt{kru@2i.tusur.ru}\\
}

\begin{document}
\maketitle

\begin{abstract}
The powers  of generating functions and its properties are analyzed. A new class of functions is introduced, based on the application of compositions of an integer $n$, called composita. The methods for obtaining reciprocal and reverse generating functions, and solutions of the functional equations $F(A(x))=G(x)$, where $A(x)$ is an unknown generating function, are proposed.   

\end{abstract}

\theoremstyle{plain}
\newtheorem{theorem}{Theorem}
\newtheorem{corollary}[theorem]{Corollary}
\newtheorem{lemma}[theorem]{Lemma}
\newtheorem{proposition}[theorem]{Proposition}

\theoremstyle{definition}
\newtheorem{definition}[theorem]{Definition}
\newtheorem{example}[theorem]{Example}
\newtheorem{conjecture}[theorem]{Conjecture}
\theoremstyle{remark}
\newtheorem{remark}[theorem]{Remark}

\newtheorem{Theorem}{Theorem}[section]
\newtheorem{Proposition}[Theorem]{Proposition}
\newtheorem{Corollary}[Theorem]{Corollary}

\theoremstyle{definition}
\newtheorem{Example}[Theorem]{Example}
\newtheorem{Remark}[Theorem]{Remark}
\newtheorem{Problem}[Theorem]{Problem}
\newtheorem{state}[Theorem]{Statement}
\makeatletter
\def\rdots{\mathinner{\mkern1mu\raise\p@
\vbox{\kern7\p@\hbox{.}}\mkern2mu
\raise4\p@\hbox{.}\mkern2mu\raise7\p@\hbox{.}\mkern1mu}}
\makeatother

\section{Introduction}

Finding expressions of reverse generating function coefficients is not an easy task and it is based on the Lagrange inversion formula \cite{Comtet_1974, Stanley_v2, Flajolet}. 
However, at the present moment there is no such a method for obtaining expressions of reverse generating function coefficients. Below we propose such a method, derived from the Lagrange inversion theorem and \emph{composita} of generating functions.    
\begin{definition}
The composita is the function of two variables defined by \cite{KruCompositae}
\begin{equation}
\label{Fnk0}F^{\Delta}(n,k)=\sum_{\pi_k \in C_n}{f(\lambda_1)f(\lambda_2) \cdots f(\lambda_k)},
\end{equation}
where $C_n$ is the set of all compositions of an integer $n$, $\pi_k$ is the composition
\begin{math}
\sum_{i=1}^k\lambda_i=n
\end{math}
 into $k$ parts exactly.
\end{definition}

Comtet \cite[ p.\ 141]{Comtet_1974} considered similar objects and identities for exponential generating functions, and called them potential polynomials. In this paper we consider the case of ordinary generating functions.

The generating function of the composita is equal to
\begin{equation}\label{GFComposita}
[F(x)]^k=\sum_{n\geq k} F^{\Delta}(n,k)x^n.
\end{equation}

For instance, we obtain the composita of the generating function $F(x,a,b)=ax+bx^2$.

The binomial theorem yields 
\begin{displaymath}
[F(x,a,b)]^k=x^k(a+bx)^k=x^k\sum_{m=0}^k \binom{k}{m}a^{k-m}b^mx^m. 
\end{displaymath}

Substituting $n$ for $m+k$, we get the following expression:
\begin{displaymath}
[F(x,a,b)]^k=\sum_{n=k}^{2k} \binom{k}{n-k}a^{2k-n}b^{n-k}x^n=\sum_{n=k}^{2k}G^{\Delta}(n,k,a,b)x^n. 
\end{displaymath}
Therefore, the composita is 
\begin{equation}
\label{Gnkab}
F^{\Delta}(n,k,a,b)=\binom{k}{n-k}a^{2k-n}b^{n-k}. 
\end{equation}

Now we show the compositae of several known generating functions \cite{Comtet_1974, Wilf_1994} in the Table \ref{tab:a}.

\begin{table}[h]
\begin{center}
\setlength\arrayrulewidth{1pt}
\renewcommand{\arraystretch}{1,3}
\begin{tabular}{|ccc|ccc|}
\hline
&\textbf{Generating function $G(x)$}& & & \textbf{Composita $G^{\Delta}(n,k)$}&\\
\hline
&$ax+bx^2$ &&& $a^{2k-n}b^{n-k}\binom{k}{n-k}$&\\ \hline
&$\frac{bx}{1+ax}$ &&&$(-1)^{n+k}\binom{n-1}{k-1}a^{n-k}b^k$&\\  \hline
&$\ln(1+x)$ &&& $\frac{k!}{n!}\genfrac{[}{]}{0pt}{}{n}{k}$& \\ \hline
&$e^x-1$  &&& $\frac{k!}{n!}\genfrac{\{}{\}}{0pt}{}{n}{k}$& \\ \hline
\end{tabular}
\caption{Examples of generating functions and their compositae}
\label{tab:a}
\end{center}
\end{table}

The notation 
\begin{math}
\genfrac{[}{]}{0pt}{}{n}{k}
\end{math}
are the Stirling numbers of the first kind (see \cite{Comtet_1974,ConcreteMath}). The Stirling numbers of the first kind  count the number of permutations of $n$ elements with $k$ disjoint cycles.

The Stirling numbers of the first kind are defined by the following generating function:
\begin{displaymath}
\psi_k(x)=\sum_{n\geq k} \genfrac{[}{]}{0pt}{}{n}{k} \frac{x^n}{n!}=\frac{1}{k!}\ln^k(1+x).
\end{displaymath}

The notation 
\begin{math}
\genfrac{\{}{\}}{0pt}{}{n}{k}
\end{math} 
are the Stirling numbers of the second kind( see \cite{Comtet_1974,ConcreteMath}). The Stirling numbers of the second kind count the number of ways to partition a set of $n$ elements into $k$ nonempty subsets.

A general formula for the Stirling numbers of the second kind is given as follows:
\begin{displaymath}
\genfrac{\{}{\}}{0pt}{}{n}{k}=\frac{1}{k!}\sum_{j=0}^k(-1)^{k-j}\binom{k}{j}j^n.
\end{displaymath}

The Stirling numbers of the second kind are defined by the generating function
\begin{displaymath}
\Phi_k(x)=\sum_{n\geq k} \genfrac{\{}{\}}{0pt}{}{n}{k} \frac{x^n}{n!}=\frac{1}{k!}(e^x-1)^k.
\end{displaymath}

Considering the formula (\ref{GFComposita}), we can conclude that the composita is a characteristic of the generating function $F(x)$. In tabular form the composita is a triangle:  
\begin{displaymath}
\begin{array}{ccccccccccc}
&&&&& F_{1,1}^{\Delta}\\
&&&& F_{2,1}^{\Delta} && F_{2,2}^{\Delta}\\
&&& F_{3,1}^{\Delta} && F_{3,2}^{\Delta} && F_{3,3}^{\Delta}\\
&& F_{4,1}^{\Delta} && F_{4,2}^{\Delta} && F_{4,3}^{\Delta} && F_{4,4}^{\Delta}\\
& \vdots && \vdots && \vdots && \vdots && \vdots\\
F_{n,1}^{\Delta} && F_{n,2}^{\Delta} && \ldots && \ldots && F_{n,n-1}^{\Delta} && F_{n,n}^{\Delta}\\
\end{array}
\end{displaymath}

\begin{example}\label{Composita_Pascal} 
Suppose 
\begin{math}
F(x)=\frac{x}{1-x}=\sum\limits_{n>0} x^n
\end{math}
where $f(0)=0$, and the rest $f(n)=1$, $n>0$. Let us find its composita. From the formula (\ref{Fnk0}) it follows that  
\begin{displaymath}
F^{\Delta}(n,k)=\sum\limits_{\pi_k \in C_n}1.
\end{displaymath}
Here summation is applied to all compositions of an integer $n$ with exact $k$ parts equal to 
\begin{math}
{n-1 \choose k-1}
\end{math}. 
Thus, 
\begin{displaymath}
F^{\Delta}(n,k)={n-1 \choose k-1}.
\end{displaymath}

Then 
\begin{displaymath}
(x+x^2+x^3+\cdots+x^n+\cdots)^k=\sum\limits_{n\geq k} {n-1 \choose k-1} x^n.
\end{displaymath}

The first terms of the generating function 
\begin{math}
F(x)=\frac{x}{1-x}
\end{math} (it is the Pascal triangle) are shown below
\begin{displaymath}
\begin{array}{ccccccccccc}
&&&&&1\\
&&&& 1 && 1\\
&&& 1 && 2 && 1\\
&& 1 && 3 && 3 && 1\\
& 1 && 4 && 6 && 4 && 1\\
1 && 5 && 10 && 10 && 5 && 1
\end{array}
\end{displaymath}
\end{example}

\section{Composition of ordinary generating functions and its composita}
For obtaining the method for finding reverse generating functions, we write the following lemma, which was proved by author \cite{Kru2010}.
\begin{lemma}
Suppose the functions $f(n)$ and $r(n)$, and their generating functions 
\begin{math}
F(x)=\sum\nolimits_{n\geq 1} f(n)x^n,
\end{math}
\begin{math}
R(x)=\sum\nolimits_{n\geq 0} r(n)x^n
\end{math} are given respectively. Then for the calculation of generating functions composition 
\begin{math}
A(x)=R(F(x))
\end{math}, the following expression is true  

\begin{displaymath}
a(0)=r(0),
\end{displaymath}
\begin{equation}\label{KruForm2}
a(n)=\sum\limits_{k=1}^nF^{\Delta}(n,k)r(k).
\end{equation}	
\end{lemma}

Further, when writing the composition 
\begin{math}
A(x)=R(F(x))
\end{math}, the following will be assumed 
\begin{math}
a(0)=r(0).
\end{math} 
 
The obtained formula (\ref{KruForm2}) defines the transformation of the sequence with the given generating function $R(x)$. In this case, the generating function of the resulting sequence will be recorded in the form of a composition 
\begin{math}
A(x)=R(F(x))
\end{math}.

Let us consider the problem of obtaining compositae for a function of the type
\begin{math}
A(x)=xR(F(x))
\end{math}, where 
\begin{math}
R(0)\ne 0
\end{math}. Let us establish the following theorem. 

\begin{theorem}\label{Theorem_CompRnkFnk}
Suppose we have generating functions
\begin{math}
R(x)=\sum\nolimits_{n\geq 0} r(n)x^n
\end{math},
\begin{math}
F(x)=\sum\nolimits_{n> 0} f(n)x^n
\end{math}, 
\begin{math}
R(n,k)
\end{math} is the expression for coefficients of the generating function $R(x)^k$,
\begin{math}
F^{\Delta}(n,k)
\end{math}
is the composita of $F(x)$.
Then, for the generating function 
\begin{math}
A(x)=xR(F(x))
\end{math}, the composita is equal to the following expression:  
\begin{equation}
A(n,m)=\begin{cases}
r(0)^n,  & n=m;\\
\sum\limits_{k=1}^{n-m} F^{\Delta}(n-m,k)R(k,m), &  n>m.\\
\end{cases}
\end{equation}
\end{theorem}

\begin{proof} 
Let us consider the expression 
\begin{math}
[B(x)]^m=\left[R(F(x)) \right]^m
\end{math}.
According to the Theorem on composition of generating functions \ref{KruForm2}, we have  
\begin{displaymath}
B(n,m)=\begin{cases}
r(0),  & n=0;\\
\sum\limits_{k=1}^{n} F^{\Delta}(n,k)R(k,m), & n>0.
\end{cases}
\end{displaymath}

Hence, for the composita of the generating function 
\begin{math}
A(x)=x\,B(x)
\end{math}, we obtain  
\begin{displaymath}
A(n,m)=B(n-m,m)=\begin{cases}
r(0)^n,  & n=m;\\
\sum\limits_{k=1}^{n-m} F^{\Delta}(n-m,k)R(k,m), & n>m.
\end{cases}
\end{displaymath}
\end{proof}

\section{Composita of a reciprocal generating function}
Reciprocal generating functions are the functions, which meet the condition \cite{Wilf_1994} 
\begin{equation}\label{Recipro1} 
A(x)B(x)=1. 
\end{equation}
Let us consider the following problem: if we know the composita 
\begin{math}
x\,B(x)
\end{math}, we need to find the composita 
\begin{math}
x\,A(x)=\frac{x}{B(x)}
\end{math}
 
Let us establish the following theorem on calculation of compositae of reciprocal generating functions.  

\begin{theorem}\label{recip_Theorem}
Suppose we have the generating function 
\begin{math}
B(x), \quad b(0)\neq 0
\end{math}
and the composita
\begin{math}
B^{\Delta}(n,k)
\end{math}
 of the generating function 
 \begin{math}
xB(x)
\end{math}. Then the composita of the generating function $xA(x)$  is equal to  

\begin{equation}
A^{\Delta}(n,m)=\begin{cases}
\frac{1}{B^{\Delta}(1,1)^m}, & n=m;\\
{{\frac{1}{B^{\Delta}(1,1)^m}\sum\limits_{k=1}^{n-m}{{{m+k-1}\choose{m-1}}\,\sum\limits_{j=1}^{k}{{{\frac{(-1)^{j}}{B^{\Delta}(1,1)^{j}}\,{{k}\choose{j}}\,B^{\Delta}(n-m+j,j)}}}}}}, & n>m. 
\end{cases}
\label{Reciprocal}
\end{equation}
\end{theorem}
  
\begin{proof} 
Based on the formula (\ref{Recipro1}), we can write 

\begin{displaymath}
xA(x)=\frac{x}{b(0)+B(x)-b(0)}=\frac{1}{b(0)}\frac{x}{1+\frac{B(x)-b(0)}{b(0)}}.
\end{displaymath}

Then for obtaining the composita, it is necessary to find the expression of generating function coefficients  
\begin{displaymath}
[xA(x)]^k=\left[\frac{1}{b(0)}\frac{x}{1+\frac{1}{b(0)}(B(x)-b(0))}\right]^k.
\end{displaymath}

Let us obtain the composita of the function 
\begin{math}
\frac{1}{b_0}(B(x)-b_0).
\end{math}
Knowing that 
\begin{math}
[xB(x)]^k = \sum\limits_{n\geq k}B^{\Delta}(n,k)x^n
\end{math}
and
\begin{math}
b(0)=B^{\Delta}(1,1),
\end{math}
we get
\begin{displaymath}
\left[\frac{1}{b_0}(B(x)-b_0)\right]^k=\sum\limits_{n\geq k}\sum_{j=1}^{k}{B^{\Delta}(1,1)^{-j}\,(-1)^{k
 -j}\,{{k}\choose{j}}\,B^{\Delta}(n+j,j)}x^n.
\end{displaymath}

The coefficients of the function 
\begin{math}
[F(x)]^k=\left[\frac{1}{B^{\Delta}(1,1)}\frac{1}{1+x}\right]^k
\end{math}
is equal to
\begin{displaymath}
\frac{1}{B^{\Delta}(1,1)^k}{n+k-1 \choose k-1}(-1)^{n}.
\end{displaymath}

Then in accordance with the Theorem \ref{Theorem_CompRnkFnk}, we obtain the desired formula  
\begin{displaymath}
A(n,m)=\begin{cases}
1, & m=0;\\
\sum\limits_{k=1}^n\sum\limits_{j=1}^{k}{B^{\Delta}(1,1)^{-j}\,(-1)^{k-j}\,{{k}\choose{j}}\,B^{\Delta}(n+j,j)}\frac{1}{B^{\Delta}(1,1)^m}{k+m-1 \choose m-1}(-1)^{k},& m>1.
\end{cases}
\end{displaymath}
\end{proof}

\section{Composita of reverse generating functions}\label{InverseComposita}
In the following lemma we give the Lagrange inversion formula, which was proved by \cite{Stanley_v2}.

\begin{lemma}[The Lagrange inversion formula]
\label{Lagrange formula}
Suppose $H(x)=\sum_{n\geq 0}h(n)x^n$ with $h(0)\neq 0$, and let $A(x)$ be defined by
\begin{equation}
A(x)=xH(A(x)).
\end{equation}
Then
\begin{equation}
n[x^n]A(x)^k=k[x^{n-k}]H(x)^n,
\end{equation}
where $[x^n]A(x)^k$ is the coefficient of $x^n$ in  $A(x)^k$ and  $[x^{n-k}]H(x)^n$ is the coefficient of $x^{n-k}$ in  $H(x)^n$.
\end{lemma}

Now, let us establish the theorem on obtaining a explicit formula for the composita of reverse generating function. 

\begin{theorem}\label{Theorem_Reversion}
Suppose we have the generating function
\begin{math}
F(x)=\sum\nolimits_{n>0} f(n)x^n, \quad f(1)\neq 0
\end{math} 
and its composita 
\begin{math}
F^{\Delta}(n,k).
\end{math}
Then the composita of the reverse generating function $A(x)$ is equal to the following expression: 

\begin{equation}
A^{\Delta}(n,m)=\begin{cases}
\frac{1}{\,F^{\Delta}(1,1)^n}, & n=m;\\
{{\frac{m}{n\,F^{\Delta}(1,1)^n}\sum\limits_{k=1}^{n-m}{{{n+k-1}\choose{n-1}}\,\sum\limits_{j=1}^{k}{{{\frac{(-1)^{j}}{F^{\Delta}(1,1)^{j}}\,{{k}\choose{j}}\,F^{\Delta}(n-m+j,j)}}}}}}, & n>m.
\end{cases}
\label{Reversion}
\end{equation}
\end{theorem}

\begin{proof} 

According to Lemma \ref{Lagrange formula}, for the solution of the functional equation $A(x)=xH(A(x))$, we can write
\begin{displaymath}
n[x^n]A(x)^k=k[x^{n-k}]H(x)^n.
\end{displaymath}

In the left-hand side, there is the composita of the generating function $A(x)$ multiplied by $n$:
\begin{displaymath}
n[x^n]A(x)^k=n\,A^{\Delta}(n,k).
\end{displaymath}

We know that
\begin{displaymath}
\left( xH(x)\right)^k=\sum_{n\geq k}G^{\Delta}(n,k)x^n.
\end{displaymath}
Then

\begin{displaymath}
\left( H(x)\right)^k=\sum_{n\geq k}G^{\Delta}(n,k)x^{n-k}.
\end{displaymath}
If we replace $n-k$ by $m$, we obtain the following expression:

\begin{displaymath}
\left( H(x)\right)^k=\sum_{m\geq 0}G^{\Delta}(m+k,k)x^{m}.
\end{displaymath}
Substituting $n$ for $k$ and $n-k$ for $m$, we get
\begin{displaymath}
[x^{n-k}]H(x)^n=G^{\Delta}(2n-k,n).
\end{displaymath}
Therefore, we get (cf. \cite{KruCompositae})
\begin{equation}\label{LagrangeCompozita}
A^{\Delta}(n,k)=\frac{k}{n}G^{\Delta}(2n-k,n).
\end{equation}

Hence, for solutions of the functional equation 
\begin{math}
A(x)=xH(A(x))
\end{math}, we can use the following expression: 
\begin{displaymath}
[A(x)]^k=\sum_{n\geq k}A^{\Delta}(n,k)x^n=\sum_{n\geq k} \frac{k}{n}G^{\Delta}(2n-k,n)x^n.
\end{displaymath}
Therefore,
\begin{displaymath}
A(x)=\sum_{n\geq 1} \frac{1}{n}G^{\Delta}(2n-1,n)x^n.
\end{displaymath}

Now we write the equation for the reverse generating functions in the form of 
\begin{math}
A(x)H(A(x))=x
\end{math}, where 
\begin{math}
H(x)=\frac{F(x)}{x}
\end{math}.
Then we can write the equation \ref{Lagrange formula} in the following form    
\begin{displaymath}
A(x)=\frac{x}{G(A(x))}.
\end{displaymath}

Hence, based on the formula \ref{LagrangeCompozita}, we get 
\begin{displaymath}
A^{\Delta}(n,m)=\frac{m}{n}R^{\Delta}(2n-m,n),
\end{displaymath}
where 
\begin{math}
R^{\Delta}(n,m)
\end{math}  is a reciprocal composita for the generating function $xH(x)$ (see the formula  \ref{Reciprocal}). 
Revealing 
\begin{math}
R^{\Delta}(n,m)
\end{math}, we obtain the desired formula.      
\end{proof}

\begin{corollary}
For the generating function 
\begin{math}
F(x)=\sum\limits_{n>0} f(n)x^n, f(1)=1,
\end{math}
the coefficients of reverse generating function is equal to the following expression
\begin{equation}\label{aReversion}
a(n)=\begin{cases}
1, & n=1;\\
{{\frac{1}{n}\sum\limits_{k=1}^{n-1}{{{n+k-1}\choose{n-1}}\,\sum\limits_{j=1}^{k}{{{(-1)^{j}\,{{k}\choose{j}}\,F^{\Delta}(n+j-1,j)}}}}}}, & n>1. 
\end{cases}
\end{equation}
\end{corollary}

Let us consider the examples.
\begin{example}
Suppose we have the generating function 
\begin{math}
F(x)=x-x^2-x^3
\end{math}, the  composita of reverse function needs to be found.  
To do this we need to find the composita of reciprocal function of
\begin{math}
R(x)=\frac{x}{1-x-x^2}
\end{math}. It has the expression \cite{KruCompositae}
\begin{displaymath}
R(n,m)=\begin{cases}
1, & n=m;\\
\sum\limits_{k=1}^{n-m} {k \choose n-m-k}{k+m-1 \choose m-1}, & n>m.
\end{cases}
\end{displaymath}

Then the composita of reverse generating function have the expression  
\begin{displaymath}
A(n,m)=\frac{m}{n}R(2\,n-m,n).
\end{displaymath}
Hence
\begin{displaymath}
A(n,m)=\begin{cases}
1, & n=m;\\
\frac{m}{n}\sum\limits_{k=1}^{n-m} {k \choose n-m-k}{k+n-1 \choose n-1}, & n>m.
\end{cases}
\end{displaymath}
\end{example}

\begin{example}\label{Example_2lnmx} 
Let us obtain the generating function 
\begin{math}
A(x)=[2\,\ln(1+x)-x]^{-1}.
\end{math}
To do this we need to find the composita of the generating function  
\begin{math}
F(x)=2\,\ln(1+x)-x.
\end{math}   
The composita of the generating function 
\begin{math}
2\,\ln(1+x)
\end{math} is equal to 
\begin{displaymath}
2^{k}\frac{k!}{n!}\genfrac{[}{]}{0pt}{}{n}{k}.
\end{displaymath}

The  composita of generating function 
\begin{math}
(-x)
\end{math} has the expression 
\begin{math}
(-1)^k\,\delta(n,k)
\end{math}. 
Then the composita of the sum of the generating functions, according to \cite{KruCompositae},  is equal to    
\begin{displaymath}
\sum_{l=0}^{k}{{{k}\choose{l}}\,\left(k-l\right)!\,2^{k-l}\,\left(-
 1\right)^{l}\,\sum_{i=l}^{n+l-k}{{{\delta_{i,l}\,\genfrac{[}{]}{0pt}{}{n-i}{k-l}}\over{\left(n-i\right)!}}}}.
\end{displaymath}

Since 
\begin{displaymath}
\delta(i,l)=\begin{cases}
1 &, l=i,\\
0 &, otherwise,
\end{cases}
\end{displaymath}
we get the following expression for the composita
\begin{displaymath}
F^{\Delta}(n,k)=\sum_{l=0}^{k}{{{{{k}\choose{l}}\,\left(k-l\right)!\,2^{k-l}\,
 \left(-1\right)^{l}\,\genfrac{[}{]}{0pt}{}{n-i}{k-l}}\over{
 \left(n-l\right)!}}}.
\end{displaymath}
 
Hence, applying the formula for the composita of the sum of the generating functions, according to \cite{KruCompositae}, we obtain 
\begin{displaymath}
{{\frac{m}{n}\,\sum\limits_{k=0}^{n-m}{{{n+k-1}\choose{n-1}}\,\sum\limits_{j=0}^{k}{\left(-
 1\right)^{j}\,{{k}\choose{j}}\,\sum\limits_{l=0}^{j}{{{{{j}\choose{l}}\,
 \left(j-l\right)!\,2^{j-l}\,\left(-1\right)^{l}\,\genfrac{[}{]}{0pt}{}{n-m-l+j}{j-l}}\over{\left(n-m-l+j\right)!}}}}}}}.
\end{displaymath} 
Therefore,  the desired formula have the expression 
\begin{displaymath}
a(n)={{\frac{1}{n}\sum\limits_{k=0}^{n-1}{{{n+k-1}\choose{n-1}}\,\sum\limits_{j=0}^{k}{\left(-1
 \right)^{j}\,{{k}\choose{j}}\,\sum\limits_{l=0}^{j}{{{{{j}\choose{l}}\,
 \left(j-l\right)!\,2^{j-l}\,\left(-1\right)^{l}\,\genfrac{[}{]}{0pt}{}{n-1-l+j}{j-l}}\over{\left(n-l+j-1\right)!}}}}}}}.
 \end{displaymath} 
\end{example}

\begin{example}\label{Example_ln1mx}
Let us find the expression of generating function coefficients 
\begin{math}
A(x)=[\ln(1+x)(1-x)]^{-1}
\end{math}.
 To do this we need to obtain the composita  of generating function 
 \begin{displaymath}
\ln(1+x)(1-x)
 \end{displaymath}
 
The composita of generating function  
\begin{math}
\ln(1+x)
\end{math} has the expression 
\begin{displaymath}
\left(-1\right)^{k}\frac{k!}{n!}\genfrac{[}{]}{0pt}{}{n}{k}.
 \end{displaymath}
The coefficients of the generating function  
\begin{math}
(1-x)^k
\end{math} are equal to  
\begin{displaymath}
{{k}\choose{n}}\,\left(-1\right)^{n}.
 \end{displaymath}
Then the composita of the product is equal to 
\begin{displaymath}
k!\,\sum_{i=k}^{n}{{\frac{1}{i!}\genfrac{[}{]}{0pt}{}{i}{k}\,{{k
 }\choose{n-i}}\,\left(-1\right)^{n-i}}}.
 \end{displaymath}

Further, applying the formula \ref{aReversion}, we get the expression for the coefficients of the reverse generating function.  
\begin{displaymath}
a(n)=\begin{cases}
1, & n=1;\\
{{\frac{1}{n}\sum\limits_{k=1}^{n-1}{{{n+k-1}\choose{n-1}}\,\sum\limits_{j=1}^{k}{{{(-1)^{j}\,{{k}\choose{j}}\,j!\,\sum\limits_{i=j}^{n+j-1}{{\frac{1}{i!}\genfrac{[}{]}{0pt}{}{i}{j}\,{{j
 }\choose{n+j-1-i}}\,\left(-1\right)^{n+j-1-i}}}}}}}}}, & n>1. 
\end{cases}
\end{displaymath}
\end{example}

\begin{example}\label{Reversion2xmExp} 
Let us find the solution for the functional equation 
\begin{math}
2\,A(x)-e^{A(x)}+1=\ln(1+x)
\end{math}. 
It can be represented in the form of 
\begin{displaymath}
A(x)=F(\ln(1+x)),
\end{displaymath}
where
\begin{math}
F(x)=[2x-e^x+1]^{-1}
\end{math}.

In order to solve this problem, it is necessary to write the expression of the function coefficients $F(x)$.
To do this we find the composita of the generating function 
\begin{math}
R(x)=2x-e^x+1
\end{math}.
 The  composita  of the generating function 
 \begin{math}
(1-e^x)
\end{math}
 is equal to 
\begin{displaymath}
\left(-1\right)^{k}\frac{k!}{n!}\genfrac{\{}{\}}{0pt}{}{n}{k}.
\end{displaymath}

The composita of generating function  
 \begin{math}
R(x)=2x
\end{math} is equal to 
\begin{math}
2^n\delta_{n,k}.
\end{math}
Hence, based on the Theorem on the sum of compositae \cite{KruCompositae}, the composita of the generating function  $F(x)$ with
\begin{math}
\genfrac{\{}{\}}{0pt}{}{0}{0}=1
\end{math}, is equal to  
\begin{displaymath}
R^{\Delta}(n,k)=\sum_{l=0}^{k}{{{k}\choose{l}}\,\left(k-l\right)!\,\left(-1\right)
 ^{k-l}\,2^{l}\,\sum_{i=l}^{n+l-k}{{{\delta_{i,l}\,\genfrac{\{}{\}}{0pt}{}{n-i}{k-l}}\over{\left(n-i\right)!}}}}.
\end{displaymath}
After transformation, we get
\begin{displaymath}
R^{\Delta}(n,k)=\sum_{l=0}^{k}{{{{{k}\choose{l}}\,\left(k-l\right)!\,\left(-1
 \right)^{k-l}\,2^{l}\,\genfrac{\{}{\}}{0pt}{}{n-l}{k-l}}\over{
 \left(n-l\right)!}}}.
\end{displaymath}

Hence, based on the formula \ref{Reversion} and with 
\begin{math}
A(n,n)=1,
\end{math} we  write 
\begin{displaymath}
F^{\Delta}(n,m)={{\frac{m}{n}\,\sum\limits_{k=0}^{n-m}{{{n+k-1}\choose{n-1}}\,\sum\limits_{j=0}^{k}{{k}\choose{j}}\,\sum_{l=0}^{j}{{{{{j}\choose{l}}\,
  \left(j-l\right)!\,2^{l}\,\left(-1\right)^{l}\,\genfrac{\{}{\}}{0pt}{}{n-m-l+j}{j-l}}\over{\left(n-m-l+j\right)!}}}}}}.
\end{displaymath}
  
Since the coefficients of reverse function have the expression
\begin{math}
a(n)=A(n,1),
\end{math} we can write the following expression  
\begin{displaymath}
f(n)={{\frac{1}{n}\,\sum\limits_{k=0}^{n-1}{{{n+k-1}\choose{n-1}}\,\sum\limits_{j=0}^{k}{{k}\choose{j}}\,\sum_{l=0}^{j}{{{{{j}\choose{l}}\,
    \left(j-l\right)!\,2^{l}\,\left(-1\right)^{l}\,\genfrac{\{}{\}}{0pt}{}{n-l+j-1}{j-l}}\over{\left(n-l+j-1\right)!}}}}}}.
\end{displaymath}

Now if we consider coefficients of the exponential generating function
 \begin{math}
n!\,f(n)
\end{math}, then we can get the sequence $A000311$ \cite{oeis}.

Therefore, from  the expression of coefficients for $F(x)$, the composita for 
\begin{math}
\ln(1+x)
\end{math}, the composition formula \ref{KruForm2}, we obtain the expression for the desired coefficients  
\begin{displaymath}
a(n)=\sum\limits_{m=1}^n {{\frac{(m-1)!}{n!}\genfrac{[}{]}{0pt}{}{n}{m}
 }} {{\sum_{k=1}^{m-1}{{{m+k-1}\choose{m-1}}\,\sum_{j=1}^{k}{{k}\choose{j}}\,\sum_{l=0}^{j}{{{{{j}\choose{l}}\,\left(j-l\right)!\,2^{l}\,\left(-1\right)^{l}\,\genfrac{\{}{\}}{0pt}{}{m-l+j-1}{j-l}}\over{\left(m-l+j-1\right)!}}}}}}.
\end{displaymath}

\end{example}

\end{document}